\documentclass{amsart} 
\usepackage{amsmath}
\usepackage{amssymb}
\usepackage{amsthm}
\usepackage[dvips]{graphicx}
\usepackage{color}
\theoremstyle{plain}
\newtheorem{theorem}{Theorem}
\newtheorem{proposition}{Proposition}
\newtheorem{lemma}{Lemma}

\newtheorem{corollary}{Corollary}
\newtheorem*{conjecture*}{Conjecture}

\theoremstyle{definition}
\newtheorem{definition}{Definition}
\DeclareMathOperator{\ver}{\mathcal{V}}
\DeclareMathOperator{\hor}{\mathcal{H}}
\DeclareMathOperator{\whg}{\widehat{\gamma}}

\DeclareMathOperator{\Hess}{Hess}
\DeclareMathOperator{\Sc}{Scal}
\DeclareMathOperator{\Ric}{Ric}
\DeclareMathOperator{\cov}{cov}
\DeclareMathOperator{\tr}{Tr}
\theoremstyle{remark}

\numberwithin{theorem}{section}
\numberwithin{lemma}{section}
\numberwithin{proposition}{section}
\numberwithin{corollary}{section}
\numberwithin{definition}{section}
\numberwithin{remark}{section}
\numberwithin{equation}{section}

\title[Semi-Riemannian manifolds with  positive curvature tensor]
{On the fundamental group of semi-Riemannian manifolds with  positive curvature tensor}

\author{Jun-ichi Mukuno}

\address{Graduate School of Mathematics, Nagoya University,
	Furocho, Chikusaku, Nagoya 464-8602, Japan}

\email{m08043e@math.nagoya-u.ac.jp}

\begin{document}
	\maketitle
	
	\begin{abstract}
		This paper presents an investigation of  
		the relation between  some positivity of the curvature  and  
		the finiteness of  fundamental groups in semi-Riemannian geometry. 
		We consider semi-Riemannian submersions $\pi : (E, g) \rightarrow (B, -g_{B}) $ under the condition 
		with   $(B, g_{B})$ Riemannian, 
		the fiber closed Riemannian, and the horizontal distribution integrable. 
		Then we prove that,
		if the lightlike geodesically complete or timelike geodesically complete semi-Riemannian manifold $E$ has some positivity of curvature,  
		then the fundamental group of the fiber is finite. 
		Moreover we construct an example of semi-Riemannian submersions with  some positivity of curvature, 
		  non-integrable horizontal distribution, and the finiteness of the fundamental group of the fiber.
	\end{abstract}
	\section{Introduction}
	This paper discusses our study of the fundamental group of semi-Riemannian manifolds with positive curvature tensor. 
	In the case of positive constant curvature, Calabi--Markus and Wolf proved the following theorem:   
	\begin{theorem}[Calabi--Markus~{\cite{Ca}} ($q=1$), Wolf~\cite{Wo}]\label{tm:Ca-Ma}
		Let $M$ be an $n$-dimensional geodesically complete
		semi-Riemannian manifold  of index $q$  
		with  constant positive curvature, 
		where $ n \geq 2q >0$.  
		Then the fundamental group $\pi_{1} (M)$ is finite. 
	\end{theorem}
	
	Kobayashi considered whether the finiteness of the fundamental group would continue to hold
	if we perturb the metric of positive constant curvature. 
	Kobayashi proposed the following conjecture: 
	\begin{conjecture*}[{Kobayashi~\cite{Kob_7}}]\label{con:Kob}
		Let $n$ and $q$ be positive integers with $n \geq 2q$. 
		Assume that $(M,\, g)$ is an $n$-dimensional  geodesically complete semi-Riemannian manifold of index $q$. 
		Suppose that we have a positive lower bound on the sectional curvature of $(M,\, g)$. 
		Then, 
		\begin{enumerate}
			\item $M$ is never compact; 
			\item if $n \geq 3$, the fundamental group of $M$ is always finite.
		\end{enumerate}
	\end{conjecture*}
	In the previous paper~\cite{MR3353737}   we  remarked that the conjecture  
	is true by using 
	Kulkarni's theorem ~\cite{MR522040} that states that
	the one-sided bound on the sectional curvature leads to constant curvature.   
	Therefore, we 
	replace the curvature condition of the conjecture by another condition.

	We study the following curvature condition of  
	Andersson and Howard~\cite{MR1664893}:  
	\begin{equation}\label{ineq : curv}
	g( R(u, v) v, u )
	\geq
	k (g(u, u) g(v, v) - g(u, v)^{2})
	\end{equation}
	for any tangent vectors $u, v$, where $R$ is the curvature tensor.
	Following Andersson--Howard~\cite{MR1664893}, we denote this condition by $R \geq k$.
	Reversing the inequality (\ref{ineq : curv}),   we write the condition as  $R \leq k$. 
	These conditions are  extensions of  
	the curvature  conditions  of Riemannian geometry that 
	the sectional curvature is bounded  above or below. 
	Alexander--Bishop~\cite{MR2425468} found 
	that these  conditions geometrically  means a signed local triangle comparison condition. 
	Under the curvature conditions of Andersson--Howard, 
	several analogues of  Riemannian comparison theory  have been obtained,  
	for instance,  gap rigidity theorem (Andersson--Howard~\cite{MR1664893}),   
	 volume comparison theorem (D{\'{\i}}az-Ramos--Garc{\'{\i}}a-R{\'{\i}}o--Hervella~
	\cite{MR2148906}), and local triangle comparison theorem (Alexander--Bishop \cite{MR2425468}).

	Our previous paper~\cite{MR3353737}  
	presents our study of an analogy of the Myers theorem in Lorentzian geometry.  
	In the current paper, we investigate an analogy of the Myers theorem in semi-Riemannian geometry. 
	A semi-Riemannian manifold $E$ is said to be \emph{lightlike} \emph{geodesically complete} 
	(resp.\  \emph{timelike} \emph{geodesically complete})
	if  any inextensible  lightlike (resp.\ timelike ) geodesic is defined on the real line.  
	We obtain the following theorem:  
	\begin{theorem}\label{5}
		Let $(E,\ g)$ be an either  lightlike geodesically complete or  timelike geodesically complete 
		semi-Riemannian manifold with $R \geq k > 0$, and  
		$(B,\ g_{B})$ a complete Riemannian manifold of dimension greater than or equal to $2$.  
		Suppose that there exists a semi-Riemannian submersion 
		$\pi: (E,\ g) \rightarrow (B,\ -g_{B})$ 
		such that the fibers are closed Riemannian manifolds, that 
		the dimension of  the fibers is greater than or equal to $2$, 
		and that the horizontal distribution is integrable. 
		Then the induced homomorphism $\pi_{*} : \pi_{1}(E) \rightarrow  \pi_{1}(B)$ is surjective and has a finite kernel.   
		Moreover, the fundamental group of the fibers is finite. 
	\end{theorem}
	We should remark that the theorem includes not only warped products but also non-warped products. 
	In fact, the semi-Riemannian product 
	$(\mathbb{H}^{l} \times S^{m}, -g_{\mathbb{H}^{l}}+g_{S^{m}})$ satisfies the assumptions of the theorem, where  
	$(\mathbb{H}^{l}, g_{\mathbb{H}^{l} } )$ is the $l$-dimensional hyperbolic space, and 
	$(S^{m}, g_{S^{m}} )$ is the $l$-dimensional sphere.  
	By the stability of the lightlike geodesic completeness by Beem--Ehrlich~\cite{MR898152}, 
	we can perturb the fiber  metric $ (S^{m}, g_{S^{m}})$ of 
	the product $(\mathbb{H}^{l} \times S^{m}, -g_{\mathbb{H}^{l}}+g_{S^{m}})$  
	with the assumption satisfied. 
	
	Note that the base manifold $(B, g_{B})$ has 
	negative curvature bounded above by $-k$ by Lemma~\ref{curvature} of the paper. 
	Hence, the theorem leads us to the following corollary:
	\begin{corollary}
		Assume that, in addition to the hypotheses of the theorem, $E$ is a closed semi-Riemannian manifold. 
		Then the fundamental group $\pi_{1}(E)$ has exponential growth. 
	\end{corollary}
	
	The  theorem gives a sufficient condition of geodesic incompleteness by its contraposition: 
	\let\temp\thetheorem
	\renewcommand{\thetheorem}{\ref{5}}
	 \begin{theorem}
		Let $(E,\ g)$ be 
		semi-Riemannian manifold with $R \geq k > 0$, and  
		$(B,\ g_{B})$ a complete Riemannian manifold of dimension greater than or equal to $2$. 
		Assume that there exists a semi-Riemannian submersion 
		$\pi: (E,\ g) \rightarrow (B,\ -g_{B})$ 
		such that the fibers are closed Riemannian manifolds, that 
		the dimension of  the fibers is greater than or equal to $2$, that the fundamental group of the fibers  is infinite, and 
		that the horizontal distribution is integrable. 
		Then $(E,\ g)$ is neither   lightlike geodesically complete nor  timelike geodesically complete.  
	\end{theorem}
	\let\thetheorem\temp
	\addtocounter{theorem}{-1}
	Let $(B, g_{B})$ be a Riemannian manifold with  negative curvature bounded above by $-k$,   
	$F$ a closed manifold of which the fundamental group is infinite, and let $ g_{F}^{*}=\{ g_{F}^{b}\}_{b \in B}$ 
	be a smooth family of Riemannian metrics of $F$. 
	By the theorem we see that  
	  $(B \times F, -g_{B}+ g_{F}^{*})$  
	  is never geodesically complete, satisfying $R \geq k > 0$. 
	 In fact, $(\mathbb{H}^{l} \times \mathbb{T}^{m}, -g_{\mathbb{H}^{l}}+ e^{2 b} g_{\mathbb{T}^{m}})$, which is 
	 a special case of  Alexander--Bishop \cite[Example7.5 (c)]{MR2425468},   
	 satisfies $R \geq 1 > 0$, but 
	 is not geodesically complete, where $b$ is a Busemann function of   $\mathbb{H}^{l}$.  

	Note that there exists a semi-Riemannian submersion with $R \geq k > 0$ and  
	 without the integrability of horizontal distribution. 
	In fact, we construct  a  semi-Riemannian example $\pi : SU(2,1)/S^{1} \rightarrow SU(2,1)/U(2)$,
	 which satisfies $R \geq k > 0$ and of which the horizontal distribution is not integrable and fiber $U(2)/S^{1}$ has a finite fundamental group. 
	This construction is motivated by the previously reported construction of the positively curved Riemannian manifold $SU(3)/S^{1}$ ~\cite{Wa}. 
	It would be interesting to determine whether we can construct new semi-Riemannian manifolds with $R \geq k > 0$  
	by using  the construction methods of  curved  Riemannian manifolds. 
	In respect of Theorem~\ref{5},  
	in general we do not know whether the theorem can be extended to semi-Riemannian submersions without the integrability of the horizontal distribution.  
	We conjecture that we expect it to be possible to remove the integrability of the horizontal distribution from the theorem. 
	\subsection*{Organization of the paper} 
	In Section $2$, we introduce some notions and propositions on 
	semi-Riemannian submersions needed for proving Theorem~\ref{5}. 
	In Section $3$, we prove Theorem~\ref{5}. 
	In Section $4$, we construct a semi-Riemannian submersion with $R \geq k >0$ 
	and  non-integrable horizontal distribution. 
	\subsection*{Acknowledgements}
	The author wishes to express his gratitude toward Kei Funano, Hiroyuki Kamada, and Shin Nayatani 
	for valuable comments. 
	This work is supported by a JSPS Grant-in-Aid for Young Scientists (B) No.\ 15K17537. 
	
	\section{Preliminary}
	A \emph{semi-Riemannian metric} $g$ of 
	a  manifold $M$ is  a  symmetric  non-degenerate $(0, 2)$ tensor. 
	If  a  manifold $M$ is equipped with  a semi-Riemannian metric $g$, $(M, g)$ is said to be a \emph{semi-Riemannian} manifold.
	In this section, we recall some terminology and results of semi-Riemannian geometry,  
	which can mainly be found in the papers by O'Neill~\cite{MR0200865, MR0216432, On}.
	Note that, although O'Neill~\cite{MR0200865, MR0216432} considered the Riemannian case, many results  are generalized to semi-Riemannian cases.
	\begin{definition}[{O'Neill~\cite[Definition 44]{On}}]
		Let $(E, g)$ and $(B, g_{B})$ be  semi-Riemannian manifolds.   
		A map $\pi : E \rightarrow B$  is a \emph{semi-Riemannian submersion}  
		if $\pi$ is a submersion satisfying the following conditions: 
		\begin{itemize}
			\item The fibers $\pi^{-1} (x) $ are semi-Riemannian manifolds for any $ x \in B$;
			\item  differential map  ${\pi_{*}}_{p} :  (T_{p} \pi^{-1} (\pi (p))^{\perp} \rightarrow T_{\pi (p)} B $ is isometric for any $p \in E$. 
		\end{itemize}
	\end{definition}
	Especially,  this research only considers the case in which the fiber and  $(B, -g_{B})$ are  Riemannian. 
	Tangent vectors normal (resp.\  tangent ) to fibers  are known as \emph{horizontal} (resp.\  \emph{vertical}) vectors. 
	A distribution $\mathcal{H}$  (resp.\ $\mathcal{V}$) over $E$ is \emph{horizontal}  (resp.\ \emph{vertical})
	if $\mathcal{H}_{p}$ (resp.\  $\mathcal{V}_{p}$) is a tangent subspace in $E$ of which the elements are horizontal (resp.\  vertical).
	Note that $T_{p} E =  \mathcal{H}_{p} \oplus  \mathcal{V}_{p}$.  
	Then for any tangent vector $v \in T_{p} E$,   $v^{\hor} \in \mathcal{H}_{p}$ and  
	$v^{\ver} \in \mathcal{V}_{p}$  are given by $v=v^{\hor}+ v^{\ver}$.  
	For a tangent vector $X$ of the base $B$, 
	a  tangent vector  $\widehat {X}$ on $E$ 
	is termed  a \emph{lift} of  $X$ 
	if $\widehat{X}$ is horizontal and $\pi_{*} \widehat{X} =X$. 
	We often identify   vectors or vector fields on $B$ with their lifts. 
	A vector field $X$ is said to be \emph{basic}  
	if  $X$ is horizontal and $\pi_{*} \widehat{X}$ is independent of   any points of  fibers. 
	Every vector field on $B$ has a unique  horizontal and basic lift on $E$.

	We denote by $\widehat \nabla$ (resp.\ $\nabla^{B}$) the Levi-Civita connection of $(E, g)$ (resp.\ $(B, g_{B})$). 
	O'Neill~\cite{MR0200865} defined the following important tensors: 
	\begin{definition}[O'Neill~\cite{MR0200865}]
		$(2, 1)$ tensor fields $T$ and $A$ on $E$ are defined by satisfying that,  for  any tangent vectors $v,\ w$ of $E$, 
		\begin{align*}
		T_{v} w &=( \widehat{\nabla}_{v^{\ver} } w^{\ver})^{\hor} + ( \widehat{\nabla}_{v^{\ver} } w^{\hor})^{\ver}; \\
		A_{v} w &=( \widehat{\nabla}_{v^{\hor} } w^{\hor})^{\ver} + ( \widehat{\nabla}_{v^{\hor} } w^{\ver})^{\hor} .
		\end{align*} 
	\end{definition}
	Note that for vertical tangent vectors $V$, $W$
	the tensor $T_{V} W$ is the second fundamental form of the fiber. 
	It is well known that the horizontal distribution is integrable if and only if $A=0$. Therefore, 
	the case we consider in this work is $A=0$. 
	The following formulas of the curvature by the tensors $T$ and $A$ hold: 
	\begin{proposition}[{O'Neill~\cite[Corollary 1]{MR0200865}}] \label{pr: ON}
		Let $\widehat K$, $K_{*}$  and $K^{\perp}$  be the sectional curvatures of $E$ and $B$, a fiber, respectively. 
		Let $X$, $Y$ be horizontal vectors, and let $V$, $W$ be   vertical vectors. 
		Then we have
		\begin{align}
		\widehat{K}(X,  Y) &= K_{*}(\pi_{*} X,  \pi_{*}  Y)-  \frac{3 g(A_{X} Y, A_{X} Y)}{g(X, X) g(Y,Y)-g(X,Y)^{2}}; \label{pr: ON-1} \\
		\widehat{K}(V,  W) &= K^{\perp}(V,  W)- \frac{g(T_{V}V, T_{W}W ) - g(T_{V} W,  T_{V} W )}{g(V, V) g(W, W)-g(V,W)^{2}}. \label{pr: ON-2}
		\end{align}
	\end{proposition}
	Let $c (t)$ be a curve in $E$ and $Z(t)$ a vector field along  $c (t)$. We consider the covariant derivative of $Z(t)$. 
	Throughout this paper, the covariant derivative of the vector field along a curve is denoted by $\prime$.  Then we have
	\begin{proposition}[{O'Neill~\cite[Theorem 1]{MR0216432}}] \label{pr: ON2}
		Let $c (t)$ and $Z(t)$ be as above. 
		The following equation holds:
		\begin{equation}
		(Z^{\prime}(t))^{\hor} =  \widehat{(\pi_{*} (Z))^{\prime} }(t) + A_{Z(t)^{\hor}}(c^{\prime} (t)^{\ver} ) +  A_{c^{\prime} (t)^{\hor}}( Z(t)^{\ver}) 
		+ T_{c^{\prime} (t)^{\ver}} (Z^{\ver} (t)). \label{eq : ON2-1}
		\end{equation}
	\end{proposition}
	O'Neill proved that  Proposition~\ref{pr: ON2} implies the following proposition,    
	which means that any geodesic in $B$ lifts to a unique horizontal geodesic in $E$: 
	\begin{proposition}[{O'Neill~\cite[Corollary 2]{MR0216432}}] \label{pr: ON3}
		Let $\pi : E \rightarrow B$ be a semi-Riemannian submersion. 
		If the initial velocity of a geodesic is horizontal, any velocity of the geodesic is also horizontal at any time. 
	\end{proposition}

	Recall 
	warped products and 
	their generalization, i.e., twisted products, 
	as important examples of semi-Riemannian submersions.  
	Let $(B, g_{B})$ and $(F, g_{F})$ be  Riemannian manifolds and $\alpha$ a smooth function of $B \times F$. 
	A semi-Riemannian manifold $(B \times F,\,  -g_{B} + e^{2 \alpha } g_{F})$,  
	is known as a  \emph{semi-Riemannian twisted product}, 
	and especially, if $\alpha$ does not depend on $B$, it is known as a  \emph{semi-Riemannian warped product}.   
	Note that the natural projection $\pi : (B \times F,\,  -g_{B} + e^{2 \alpha } g_{F}) \rightarrow (B, -g_{B})$ is a semi-Riemannian submersion.  
	We denote by $\widehat\nabla$, $\nabla^{B}$, and $\nabla^{F}$  the Levi-Civita connections of $(B \times F,\,  -g_{B} + e^{2 \alpha } g_{F})$, 
	$(B,\,g_{B})$, and 
	$(F,\,g_F)$, respectively. 
	The natural projection from $B \times F$ to $F$ 
	is denoted by $\pi_{F}$.   
	For a vector field $Z$ on $B$ (resp.\ $F$), 
	a vector field $\widehat{Z}$ on the product manifold $B \times F$ 
	is  a \emph{lift} of  $Z$ 
	if ${\pi_{F}}_{*} (\widehat{Z})=0$ and ${\pi}_{*} (\widehat{Z})=Z$ 
	(resp.\ ${\pi}_{*} (\widehat{Z})=0$ and ${\pi_{F}}_{*}  (\widehat{Z})=Z$).   
	We have the following formula of the mean curvature vector of the fibers in warped products:  
	\begin{proposition}[O'Neill~{\cite[Chapter 7, Proposition 35]{On}}(warped products case), 
		Chen~{\cite[Chapter VII, Proposition 1.2]{MR627323}(twisted product cases)}]\label{pr:na}
		For any lifts   $\widehat{U}$, $\widehat{V}$ of vector fields $U$, $V$  on $F$  in $(B \times F,\,  -g_{B} + e^{2 \alpha } g_{F})$, 
		\begin{equation}
		T_{\widehat{U}}\widehat{V} =
		e^{2 \alpha } g_{F}(U,V)\nabla^{B} \log \alpha. 
		\label{mean curv}  
		\end{equation}
	\end{proposition}

	We use the following equations on geodesics in warped products later:  
	\begin{proposition}[O'Neill~{\cite[Chapter 7, Proposition 38]{On}}]\label{geodesic}
		Any geodesic $\gamma(t)=(\gamma_{B}(t),  \gamma_{F}(t))$  in the warped product $(B \times F,\,  -g_{B} + e^{2 \alpha } g_{F})$ 
		satisfies the following two conditions: 
		\begin{align*}
		\nabla^{B}_{\partial/ \partial t} \gamma_{B}^{\prime} (t) &= 
		-  e^{2 \alpha}   g_{F}(\gamma_{F}^{\prime}(t),\, {\gamma_{F}}^{\prime}(t)) 
		\nabla^{B} \alpha; \\
		\nabla^{F}_{\partial/\partial t} \gamma_{F}^{\prime}(t)
		&=-2 \frac{d\alpha (\gamma_{B}(t))}{dt} 
		{\gamma_{F}}^{\prime}(t). 
		\end{align*}
	\end{proposition}

	\section{Proof of Theorem~\ref{5}}
	Let us begin the proof of Theorem~\ref{5}. 
	First, we prove the following lemma: 
	\begin{lemma}\label{curvature}
		$(B,\ g_{B})$ has negative curvature bounded above by $-k$. 
	\end{lemma}
	\begin{proof} 
		We assume that $A=0$. 
		By using the equation~(\ref{pr: ON-1}) of  Proposition~\ref{pr: ON}, we have
		$K(X, Y )= K_{*}(\pi_{*} X,  \pi_{*}  Y).$ 
		The curvature condition $R \geq k > 0$ implies $K_{*}(\pi_{*} X,  \pi_{*}  Y)$. 
		Let $K_{B}$ be the curvature of $(B,\ g_{B})$.  
		Since $K_{*}(\pi_{*} X,  \pi_{*}  Y)=- K_{B}(\pi_{*} X,  \pi_{*}  Y)$, 
		the curvature of $(B,\ g_{B})$ is bounded above by $-k$. 
	\end{proof} 
	Let $b_{0}$ be a fixed point of $B$, and $F$ the fiber $\pi^{-1}(b_{0})$.   
	It follows that the universal covering space of $B$ is contractible by 
	the Hadamard--Cartan theorem. 
	By the following exact sequence of the homotopy group
	$$
	1 \rightarrow \pi_{1}(F) \rightarrow \pi_{1}(E) \xrightarrow{\pi_{*}} \pi_{1}(B) \rightarrow 1,
	$$
	we see that  $\pi_{*}$ is surjective and that the kernel of $\pi_{*}$  is isomorphic to $\pi_{1}(F)$.   
	Therefore, it is sufficient to prove that the fundamental group of $F$ is finite. 
	Moreover, we assume that $B$ is simply connected and contractible throughout the proof. 
	In fact, this is because, for the universal covering $\cov : \widetilde B \rightarrow B$,  
	the fiber of the induced submersion ${\cov}^{*}E \rightarrow  \widetilde B$ is  the same  as the original fiber of $\pi : E \rightarrow   B $.

	Here we investigate the metric structure of  the entire space $E$. 
	We can define the projective map $\pi_{F} : E \rightarrow F$  by the following: 
	For any $p \in E$, we have a unique geodesic $\gamma$ in $B$ from 
	$\pi ( p ) $ to $ b_{0}$.   
	Let $\widehat{\gamma}$ be the horizontal lift of $\gamma$. 
	Then $\pi_{F}(p) $  is given by the end point of  $\widehat{\gamma}$.  
	Note that   $\pi_{F}$ is surjective and smooth since geodesics smoothly depend  on initial points. 
	As the horizontal distribution is integrable, 
	a horizontal manifold $\pi^{-1}(f)$ is diffeomorphic to  $B$ under $\pi$. 
	Then we see that  $\phi=(\pi, \pi_{F} ) :  E \rightarrow B \times F$ is a diffeomorphism.
	By the construction of $\pi_{F}$, the fiber $\pi_{F}^{-1} (f) $ diffeomorphic to $B \times \{ f \} $ is a horizontal manifold. 
	The definition of  semi-Riemannian submersion determines that  $\pi :  \pi_{F}^{-1} (f)  \rightarrow B$ is isometric. Therefore, 
	we obtain the following lemma: 
	\begin{lemma}
	Set  $g_{F}^{b} = g|_{\pi^{-1}(b)}$ and  a  smooth family of the Riemannian metric $g_{F}^{*}=\{g_{F}^{b} \}_{b \in B}$  of $F$ with respect to the points of $B$. 
	Then $(E, g)$ is isometric to $(B \times F, -g_{B} + g_{F}^{*} )$.  
	\end{lemma}
	Throughout the proof, $(E, g)$ is regarded as $(B \times F, -g_{B} + g_{F}^{*} )$. 
	This metric structure leads us 
	to define  lifts of any vector field on fibers  in $E$ as well as warped products and twisted products. 
	Any horizontal lift of vector fields on the base space does not depend on the fibers.

	Next, we investigate the curvature of the fiber. 
	Parallel vector fields along horizontal geodesics preserve verticality. 
	
	\begin{lemma}\label{le:parallel}
		Let $c: \mathbb{R} \rightarrow E$ be a horizontal curve and 
		$V(t)$  a parallel vector field along $c(t)$ with $V(0)$ vertical. 
		Then $V(t)$ is a vertical vector field.  
	\end{lemma} 
	\begin{proof}
		Since $A=0$ and $c(t)$ is a horizontal curve, 
		the equation~(\ref{eq : ON2-1}) gives 
		$(V^{\prime}(t))^{\hor} =  \widehat{( \pi_{*} (V(t)))^{\prime}} $. 
		As $V(t)$ is parallel, $V^{\prime}(t)=0$.   It follows that  
		$(\pi_{*} (V(t)))^{\prime} =0$.
		Since  $\pi_{*} V(0)=0 $, 
		we have  $\pi_{*} V(t)=0 $. 
		Therefore, $V(t)$ is vertical.  
	\end{proof}

	Place any point $p_{0} \in E$ and unit vertical vector $V_{0} \in \mathcal{V}_{p_{0}}$. 
	Set $x_{0} = \pi (p_{0})$. 
	Let $u_{0}$  be any unit tangent vector at $x_{0}$ and 
	$\gamma$ the geodesic in $B$ starting from $x_{0}$ with initial velocity $u_{0}$. Further, let   
	$N$ be the  $C^{1}$   gradient vector  field  of the Busemann function $b_{\gamma}$ of $B$ associated to $\gamma$.
	Note that the flow of $N$ is a geodesic in $B$. 
	Then the Busemann function $b_{\gamma}$ is extended to  the entire space $E = B \times F$ by 
	$B \times F \ni (b, f) \mapsto  b_{\gamma}(b) \in \mathbb{R}$. 
	This extended Busemann function is denoted by $\widehat{b_{\gamma}}$. 
	We write the gradient vector of $\widehat{b_{\gamma}}$ as
	 $\widehat{N}$,  which is the horizontal lift of $N$. 
	Then 
	for any $p \in E$
	we define a map $S : \mathcal{V}_{p} \rightarrow  T_{p} E$ 
	by $S(V_{0}) =\widehat{\nabla}_{V_{0}} \widehat{N} $ for any  vertical vector  $V_{0} \in \mathcal{V}_{p}$. 
	Then we have 
	\begin{lemma}
		$S(V_{0})$ is vertical. Moreover, $S(V_{0})$  has no term of differentials of $N$. 
	\end{lemma}
	\begin{proof}
		For $p = (b, f)$, 
		take a coordinate  
		$(x_{1}, x_{2}, \ldots, x_{i}, \ldots,  x_{\dim B})$  
		of  $B \times \{ f\}$
		and  a coordinate  $(y_{1}, y_{2}, \ldots, y_{k}, \ldots,  y_{\dim E - \dim B})$
		of   $\{ b\} \times F$ around $p$.  
		Then we have
		\begin{equation*} 
		V_{0}=\sum_{k}  V_{0}^{k} \frac{\partial}{\partial y_{k}}, \quad \widehat{N}=\sum_{i} \widehat{N}^{i} \frac{\partial}{\partial x_{i}}.
		\end{equation*} 
		Note that  $\partial \widehat{N}^{i} / \partial y_{k} =0$ since
		$\widehat{N}$ is a horizontal lift and 
		$E$ is the product manifold $B \times F$. 
		Therefore, 
		\begin{equation*}
		S(V_{0})=\widehat{\nabla}_{V_{0}} \widehat{N} 
		=\sum_{k}   V_{0}^{k} \widehat{\nabla}_{ \frac{\partial}{\partial y_{k}}} \left(\sum_{i} \widehat{N}^{i} \frac{\partial}{\partial x_{i}} \right)
		=\sum_{k, i }  \widehat{N}^{i}   V_{0}^{k} \widehat{\nabla}_{ \frac{\partial}{\partial y_{k}}} \frac{\partial}{\partial x_{i}}. 
		\end{equation*}
		Moreover,  we have
		\begin{equation*}
		\widehat{\nabla}_{ \frac{\partial}{\partial y_{k}}} \frac{\partial}{\partial x_{i}}
		=\sum_{l, m } \frac{g^{l m}}{2} \frac{\partial g_{k l}}{\partial x_{i}} 
		\frac{\partial}{\partial y_{m}}, 
		\end{equation*}
		where $l$, $m$ are the index of the coordinate of the fiber. 
		Hence, $S(V_{0})$ is vertical and  $S(V_{0})$  has no term of differentials of $N$.   
	\end{proof}

	We  can extend the vertical vector $V_{0}$  to  a vertical vector field on $E$ by using partitions of unity, and 
	restrict this vector field to the $C^{2}$-submanifold $\widehat{b_{\gamma}}^{-1}(\widehat{b_{\gamma}}(p_{0}))$.
	The restricted vector field stands for $V_{0}$ by abuse of notation.  
	We consider the differential equation 
	$\widehat{\nabla}_{\widehat{N}} V =0$ with the initial value $V(p)=V_{0}(p)$ for any $p \in \widehat{b_{\gamma}}^{-1}(\widehat{b_{\gamma}}(p_{0}))$. 
	Take any  flow $\tau (t)$ of $\widehat{N}$ with $\tau(0) \in  \widehat{b_{\gamma}}^{-1}(\widehat{b_{\gamma}}(p_{0}))$ and 
	 $\tau (0)=V_{0}(\tau (0))$. Then  $\widehat{\nabla}_{\widehat{N}} V = \widehat{\nabla}_{\tau^{\prime}(t)} V(\tau (t)) =0$ on the geodesic $\tau (t)$.  
	 Then $V(\tau (t))$ is  solved and vertical by Lemma~\ref{le:parallel}.  
	 By collecting $V(\tau (t))$,  we see that $V$ is  $C^{2}$ vertical vector fields on $E$. 
	Let  $\widehat{\gamma(t)}$ be  the horizontal lift of $\gamma(t)$ starting from $p_{0}$. 
	We denote  the tangent vector $S (V)  (\whg (t) )$ at $\whg (t)$ by $S_{t} (V)$. 
	Set  $ h(t) = g(S_{t}(V), V)$.  
	We should remark that 
	$$
	g(S_{t}(V), V)= g(\widehat{\nabla}_{V} \widehat{N}, V )=-g(\widehat{N}, \widehat{\nabla}_{V} V).
	$$ 
	We have 
	\begin{align*}
	\frac{d}{dt} h(t)&= - \widehat{N} g(\widehat{N}, \widehat{\nabla}_{V} V) \\
	&= - g(\widehat{\nabla}_{\widehat{N}} \widehat{N}, \widehat{\nabla}_{V} V )-g( \widehat{N},  \widehat{\nabla}_{\widehat{N}}  \widehat{\nabla}_{V} V ) \\
	&=- g( \widehat{N},  \widehat{\nabla}_{\widehat{N}}  \widehat{\nabla}_{V} V ) \\
	&=- g( \widehat{N}, R (\widehat{N}, V)V) -  g( \widehat{N},  \widehat{\nabla}_{V}  \widehat{\nabla}_{ \widehat{N} } V )- g( \widehat{N},  
	\widehat{\nabla}_{[\widehat{N}, V] } V )  \\
	&=- g( \widehat{N}, R (\widehat{N}, V)V) + g( \widehat{N}, \widehat{\nabla}_{\widehat{\nabla}_{V} \widehat{N} } V )  \\
	&=- g( \widehat{N}, R (\widehat{N}, V)V) + g( \widehat{N}, [\widehat{\nabla}_{V} \widehat{N} , V ] + \widehat{\nabla}_{V}   \widehat{\nabla}_{V} \widehat{N} )  \\
	&=- g( \widehat{N}, R (\widehat{N}, V)V) + g( \widehat{N}, \widehat{\nabla}_{V}   \widehat{\nabla}_{V} \widehat{N} )  \\
	&=- g(R (V, \widehat{N})\widehat{N}, V ) - g(\widehat{\nabla}_{V}\widehat{N}, \widehat{\nabla}_{V} \widehat{N} ) \\
	&\leq -k g(V, V) g(\widehat{N} , \widehat{N}) - g(\widehat{\nabla}_{V}\widehat{N}, V)^{2}  \\
	&\leq k -h(t)^{2}.
	\end{align*}
	By Ricatti's argument, 
	we obtain $|h(t)| \leq \sqrt{k}$. 
	We should remark that 
	$$
	|h(t)|=
	|g(\widehat{N}, \widehat{\nabla}_{V}V )|=  |g(\widehat{N},  T_{V} V)| =
	|g(\widehat{N},  T_{V} V)| 
	$$
	as $h(t)$ is defined on $\mathbb{R}$.
	Since $|h(0)|= |g(u_{0},  T_{V_{0}} V_{0})|$. 
	Here $|V|$ represents $\sqrt{|g (V ,V )|}$.   
	If $|T_{V_{0}} V_{0}|\neq 0$,  
	we take $u_{0}$ as  $\frac{T_{V_{0}} V_{0}}{|T_{V_{0}} V_{0}|}$. 
	Then we obtain $|T_{V_{0}} V_{0}| \leq \sqrt{k}$.
	Therefore, we have confirmed the following lemma: 
	\begin{lemma}\label{le:bded}
		$|T_{V} V| \leq \sqrt{k}$ for any vertical unit tangent vector $V$.  
	\end{lemma}

	Hereafter, we calculate the curvature of the fiber. 
	\begin{proposition}
		The curvature of fibers is non-negative. 
	\end{proposition}
	\begin{proof}
		Equation~(\ref{pr: ON-2}) of  Proposition~\ref{pr: ON} implies that for any unit vertical vector $U$ and $V$, 
		\begin{align*}
		\widehat{K}(U, V)&=K^{\perp}(U,V) - 
		\frac{g(T_{U} U, T_{V} V) - g(T_{U} V, T_{U} V)}{ g(U, U) g(V, V) - g(U, V)^{2} } \\
		& \leq K^{\perp}(U,V) - \frac{g(T_{U} U, T_{V} V) }{ g(U, U) g(V, V) - g(U, V)^{2} } \\
		& \leq K^{\perp}(U,V) + \frac{ |T_{U} U ||T_{V} V| }{ g(U, U) g(V, V) - g(U, V)^{2} }. \\
		\end{align*}
		Since $\widehat{K}(U, V) \geq k $,   we have
		\begin{align*}
		K^{\perp}(U,V) & \geq 
		k-  \frac{ |T_{U} U ||T_{V} V| }{ g(U, U) g(V, V) - g(U, V)^{2} } \\ 
		&\geq k-(\sqrt{k})^{2}=0. 
		\end{align*}
	\end{proof}
	
	Therefore,  
	the following structural theorem for the fundamental group of  a closed  Riemannian manifold of non-negative curvature 
	leads us to the restriction of the topology of the fibers: 
	\begin{theorem}[Toponogov~\cite{MR0108808}, {Cheeger--Gromoll~\cite[Theorem 3]{MR0303460}}]\label{tm:split}
		Let $M$ be a closed Riemannian manifold of non-negative sectional curvature.
		Then, the universal covering Riemannian manifold $\widetilde{M}$ of $M$  can be split isometrically as $\mathbb{R}^{l} \times \widetilde{N}$,
		where $\widetilde{N}$ is a closed Riemannian manifold.
		Moreover, the fundamental group $\pi_{1}(M)$ includes a free abelian  subgroup $\mathbb{Z}^{l}$ of
		finite index that acts properly discontinuously and co-compactly   
		as a  deck transformation on the Euclidean factor.
	\end{theorem}

	From Theorem~\ref{tm:split}, it follows that the universal covering Riemannian manifold  of the fiber 
	is the Riemannian product
	manifold of the Euclidean space and some closed Riemannian manifold.  
	It is sufficient to prove that the dimension of the Euclidean factor is zero.
	Suppose, by way of contradiction, that this dimension is not zero.
	Then, the fundamental group of the fiber has a free abelian normal subgroup $\mathbb{Z}^{l}$ of finite index for some  
	$l > 0$. 
	In the previous paper~\cite{MR3353737}, which considered the Lorentzian case, 
	we used the Penrose singularity theorem. However, we are aware of few semi-Riemannian analogies of the Penrose singularity theorem. 
	Therefore, we need another strategy. 
	First we show the following proposition: 
	\begin{proposition}\label{prop : str}
		If $l>0$, the entire space $E$ admits 
		a structure of warped products.
	\end{proposition}
	To prove the proposition, we prove several lemmas. 
	First, 
	we check that  the fibers  are  totally umbilical. 
	\begin{lemma}
		For any $b \in B$, the fiber 
		$\pi^{-1} (b)$ is a totally umbilical submanifold with constant mean curvature. 
	\end{lemma}

	\begin{proof}
		Let $\nabla^{\perp}$ be the Levi-Civita connection of  the fibers 
		$\pi^{-1} (b)$ with the induced metric and 
		$K^{\perp}$ the  sectional curvature  of $\nabla^{\perp}$.  
		Theorem~\ref{tm:split} implies that
		there exists a closed Riemannian manifold $S$ such that 
		the universal covering  of 
		$\pi^{-1} (b)$ is isometric to $ S \times \mathbb{R}^{l}$.   
		Take $p \in \pi^{-1} (x)$. We have $T_{p} \pi^{-1} (x) \simeq   T_{\tilde{p}}S + T_{\tilde{p}} \mathbb{R}^{l}$, 
		where $\tilde{p}$ is a lift of $p$.  Hence, we identify $T_{p} \pi^{-1} (b)$ with $ T_{\tilde{p}}S + T_{\tilde{p}} \mathbb{R}^{l}$. 
		For any unit tangent vector  $U_{1} \in T_{\tilde{p} }  S $, $U_{2} \in T_{\tilde{p}} \mathbb{R}^{l}$ with $g(U_{1}, U_{2})=0$,   we have 
		$$
		K^{\perp} (U_{1}, U_{2})=0.
		$$ 
		By the equation~(\ref{pr: ON-2}),  we obtain 
		\begin{align*}
		0&=K^{\perp}(U_{1}, U_{2}) \\
		&=
		\widehat{K}(U_{1},  U_{2}) +g(T_{U_{1}}U_{1}, T_{U_{2}}U_{2} ) - g(T_{U_{1}} U_{2},  T_{U_{1}} U_{2} ) \\
		&\geq k + g(T_{U_{1}}U_{1}, T_{U_{2}}U_{2} )  \\
		&\geq k  -  |T_{U_{1}} U_{1}| |T_{U_{2}} U_{2} |. 
		\end{align*}

		Lemma~\ref{le:bded} implies that 
		the right-hand side is non-negative. 
		Therefore, we have $ k-  |T_{U_{1}} U_{1}| |T_{U_{2}} U_{2} | =0$. 
		Since the equality of the Cauchy--Schwartz inequality  holds, 
		$T_{U_{1}} U_{1} =  T_{U_{2}} U_{2} =\sqrt{k} \nu(U_{1}, U_{2})$, where $\nu(U_{1}, U_{2}) $ is a horizontal unit vector. 
		Moreover, $T_{U_{1}} U_{2} =0$ holds. 
		We see  that $\nu= \nu(U_{1}, U_{2}) $ is independent of the choice of $U_{1} \in T_{\tilde{p} }  S $, $U_{2} \in T_{ \tilde{p} } \mathbb{R}^{l}$.  
		For any $U =U_{1} + U_{2} \in T_{p } \pi^{-1}(b)$,
		\begin{align*}
		T_{U}U&=T_{U_{1}}U_{1} + 2T_{U_{1}} U_{2} + T_{U_{2}}U_{2}  \\
		&=\sqrt{k}(g(U_{1}, U_{1}) +g(U_{2}, U_{2})  )\nu=\sqrt{k}(g(U, U)  )\nu. 
		\end{align*}
		For any $U , V \in T_{p } \pi^{-1}(b)$ with $g(U, V)=0$, 
		\begin{align*}
		T_{U} V &=\frac{T_{U+V} (U+V) - T_{U} U -T_{V}V}{2} \\
		&=\frac{\sqrt{k}}{2}(g(U+V, U+V) -g(U,U)-g(V,V) ) \nu =0.
		\end{align*}
		Therefore, we obtain $T_{U} V =\sqrt{k} g(U,V) \nu$ such that $g(\nu, \nu)=-1$. 
	\end{proof}

	We prove the following lemma: 
	\begin{lemma} 
		$E$ has the structure of twisted product $(B \times F, -g_{B}+ e^{2 \alpha } g_{F})$, where $\alpha$ is  a function on $E$ and 
		$g_{F}$ is some Riemannian metric of $F$. 
	\end{lemma}
	\begin{proof}

		Note that $\nu$ is  a unit horizontal vector field. Recall that $E=(B\times F, -g_{B}+ g_{F}^{*})$.   
		For any point $p=(b, f) \in  B\times F =E$, 
		let $U$ be  any tangent vector in $T_{f}F$ and let $\widehat{U}$ be the vertical lift of $U$ to $E$. 
		Note that $\widehat{U}$  is the vertical vector field along the horizontal submanifold $B \times \{ f\}$. 
		For fixed $f \in F$, 
		we can define $h_{U}(b) = g_{F}^{b}(f)(\widehat{U}, \widehat{U}) $.  
		Let $X$ be any vector field  on $B$ and let $\widehat{X}$ be its lift. 
		Then 
		we have
		\begin{align*}
		X h_{U}(b) &=X g_{F}^{b}(\widehat{U}, \widehat{U}) = \widehat{X} g(\widehat{U}, \widehat{U}) 
		= 2 g(\widehat{U}, \widehat{\nabla}_{ \widehat{X} } \widehat{U})
		= 2 g(\widehat{U}, \widehat{\nabla}_{ \widehat{U} } \widehat{X})  \\
		&=
		2 (  \widehat{U} g(\widehat{U},  \widehat{X})  - g(\widehat{\nabla}_{ \widehat{U} } \widehat{U}, \widehat{X}) 
		) 
		= - 2 g(\widehat{\nabla}_{ \widehat{U} } \widehat{U}, \widehat{X}) 
		\\
		&= - 2 g(T_{\widehat{U} } \widehat{U} , \widehat{X} ) 
		= -  2 \sqrt{k}  g(\widehat{U} ,\widehat{U} ) g(\nu , \widehat{X} ) 
		=   -   2 \sqrt{k}  g(\nu , \widehat{X} ) h_{U}(b).
		\end{align*}
		Set 
		$$H_{U}(b) =   \frac{\log h_{U}(b)}{2 \sqrt{k}}.$$ 
		Then we obtain 
		$$
		g_{B}(X, \nabla^{B} H_{U} )   = g_{B}(X, \nu). 
		$$
		Therefore, on any $B \times \{ f\}$, we have  $\nabla^{B} H_{U}  = \nu $.
		It follows that $\nabla^{B} H_{U}$ does not depend on $U$. 
		
		Take a fixed point $b_{0} \in B$. Since the fibers have torus factors, 
		there exists  a unit vector field along  $\pi^{-1}(b_{0})$,  
		denoted by $U^{\prime}$. Let $\widehat{U^{\prime}}$ be the vertical lift of  $U^{\prime}$ to $E$. 
		Note that $\widehat{U^{\prime}}$ is the vertical vector field on the entire space $E$. 
		We know that
		$\nabla^{B} (H_{U}- H_{U^{\prime}}) =0$. Hence, 
		$H_{U}- H_{U^{\prime}}$  does not depend on $B$. It follows that 
		$H_{U}(b)- H_{U^{\prime}}(b) =H_{U}(b_{0})- H_{U^{\prime}}(b_{0})$ for any $b \in B$. Then we have 
		\begin{equation*}
		g_{F}^{b}(f)(\widehat{U}, \widehat{U})=h_{U}(b) = 
		\frac{h_{U^{\prime}} (b)}{h_{U^{\prime}} (b_{0})} h_{U}(b_{0})= 
		\frac{h_{U^{\prime}} (b)}{h_{U^{\prime}}(b_{0})} g_{F}^{b_{0}}(f)(\widehat{U}, \widehat{U}). 
		\end{equation*} 
		We denote  by $\alpha (b, f)$ the function 
		$\frac{1}{2} \log \frac{h_{U^{\prime}} (b)}{h_{U^{\prime}}(b_{0})} $, 
		by $F$ the fiber $\pi^{-1}(b_{0})$, 
		and by $g_{F} $ the metric  $g_{F}^{b_{0}}(f)$ of $F$. 
		We obtain  $g_{F}^{b}(f)(\widehat{U}, \widehat{U}) = e^{2 \alpha } g_{F}(\widehat{U}, \widehat{U})$. 
		We have proved  that $E$ is a twisted product. 
	\end{proof}

	Note that $F = \mathbb{T}^{l} \times S$ by the splitting theorem, where $S$ is a closed Riemannian manifold. 
	It follows that $E=(B\times \mathbb{T}^{l} \times S, -g_{B}+e^{2 \alpha}g_{\mathbb{T}^{l}} +e^{2 \alpha} g_{S})$, where 
	$S$ is some closed manifold without torus factors and $g_{\mathbb{T}^{l}}$ is the flat metric of the torus. 
	First, we consider the case $l \geq 2$. 
	Then we have
	\begin{lemma}\label{le : curv}
		The submanifold $(\mathbb{T}^{l}, e^{2 \alpha}g_{\mathbb{T}^{l}})$  has non-negative curvature. 
	\end{lemma}
	\begin{proof}
		Let $\Pi$ be the fundamental form of $(\mathbb{T}^{l}, e^{2 \alpha}g_{\mathbb{T}^{l}})$ in the fiber 
		$(\mathbb{T}^{l} \times S, e^{2 \alpha} g_{\mathbb{T}^{l}} +e^{2 \alpha} g_{S})$. 
		Note that the fiber has non-negative curvature. 
		Take any tangent vectors $U, V$ of $\mathbb{T}^{l}$. 
		Then 
		\begin{equation}\label{mean:curv}
		\Pi(U, V)=- g_{\mathbb{T}^{l}}(U, V) (\nabla^{F} \alpha)^{\perp}, 
		\end{equation}
		where $\nabla^{F}$ is the Levi-Civita connection of $(F, g_{F})$ and 
		$(\nabla^{F} \alpha)^{\perp} $ is the component of  $\nabla^{F} \alpha$ perpendicular to the torus. 
		By using the Gaussian equation of the torus in the fiber, we see that the sectional curvature  of 
		$(\mathbb{T}^{l}, e^{2 \alpha}g_{\mathbb{T}^{l}})$  is  non-negative.  
	\end{proof}

	Lemma \label{le : curv} implies the following lemma: 
	\begin{lemma}\label{le:torus}
		$\alpha$ does not depend on the torus.  
	\end{lemma}
	\begin{proof}
		Note that 
		the scalar curvatures $\Sc_{(\mathbb{T}^{l},  g_{\mathbb{T}^{l}})}$ and 
		$\Sc_{(\mathbb{T}^{l}, e^{2 \alpha} g_{\mathbb{T}^{l}})}$ 
		of $(\mathbb{T}^{l}, g_{\mathbb{T}^{l}})$ and  $(\mathbb{T}^{l}, e^{2 \alpha} g_{\mathbb{T}^{l}})$ are 
		zero and non-negative, respectively.   
		Moreover, we have the following formula on scalar curvature: 
		\begin{equation}
		\Sc_{(\mathbb{T}^{l}, e^{2 \alpha } g_{\mathbb{T}^{l}})}=e^{-2 \alpha}(\Sc_{(\mathbb{T}^{l},  g_{\mathbb{T}^{l}})}
		+ 2(l-1) \Delta_{(\mathbb{T}^{l},  g_{\mathbb{T}^{l}})} \alpha - (l-2)(l-1)|d \alpha |^{2}), \label{scal}
		\end{equation}
		where $\Delta_{(\mathbb{T}^{l},  g_{\mathbb{T}^{l}})}$  is the Laplacian of $(\mathbb{T}^{l},  g_{\mathbb{T}^{l}})$.   
		First, we consider the case $l=2$. 
		From the equation~(\ref{scal})   
		it follows that 
		\begin{equation*}
		\Sc_{(\mathbb{T}^{l}, e^{2 \alpha } g_{\mathbb{T}^{l}})}=2 e^{-2 \alpha}  \Delta_{(\mathbb{T}^{l},  g_{\mathbb{T}^{l}})} \alpha . 
		\end{equation*}
		Since $\Sc_{(\mathbb{T}^{l}, e^{2 \alpha } g_{\mathbb{T}^{l}})}$ is non-negative,  $\Delta_{(\mathbb{T}^{l},  g_{\mathbb{T}^{l}})} \alpha \geq 0 $ on the torus.  
		By the maximal principle, $\alpha$ is constant on the torus. 
		Next, we consider the remaining case $l \geq 3$. 
		Then, by equation~(\ref{scal}) we obtain  the following inequality 
		\begin{equation*}
		2 \Delta_{(\mathbb{T}^{l},  g_{\mathbb{T}^{l}})} \alpha \geq  (l-2) |d \alpha |^{2} \geq 0. 
		\end{equation*}
		The maximal principle implies that $\alpha$ is constant on the torus. The proof is complete. 
	\end{proof}

	Therefore, we have
	\begin{lemma}
		$\alpha$ does not depend on the fiber.  
	\end{lemma}
	\begin{proof}
		Lemma~\ref{le:torus} implies that the sectional curvature of 
		$(\mathbb{T}^{l}, e^{2 \alpha } g_{\mathbb{T}^{l}})$ is zero. 
		Therefore, by equation~(\ref{mean:curv}) and the Gaussian equation, we obtain $ (\nabla^{F} \alpha)^{\perp} =0$. 
		Since $\alpha$ does not depend on the torus by Lemma~\ref{le:torus}, we have $ \nabla^{F} \alpha =0$. 
		The lemma has been proved. 
	\end{proof}

	Next, we consider the case of $l=1$, namely $F_{0}= \mathbb{T}^{1} \times S$, where $S$ is a closed Riemannian manifold. 
	Since $S$ is closed, there exists a closed geodesic $S^{1}$ in $(S, g_{S})$. Therefore, we have 
	the natural immersion 
	$\iota :  \mathbb{T}^{1} \times S^{1}=\mathbb{T}^{2} \rightarrow  \mathbb{T}^{1} \times S$. 
	Let the curvature of $(\mathbb{T}^{2}, e^{2 \alpha \circ \iota} g_{\mathbb{T}^{2}})$ be $\overline{S}$. 
	Then 
	\begin{equation*}
	\overline{S} =2e^{-2 \alpha} \Delta_{(\mathbb{T}^{2}, g_{\mathbb{T}^{2}})}(\alpha \circ \iota).
	\end{equation*}
	Since the fiber has non-negative curvature, $\overline{S} \geq 0$ by using the Gauss equation.  
	The maximal principle implies that $\alpha$ does not depend on 
	the image of 
	$(\mathbb{T}^{1} \times S^{1})$ under the immersion $\iota$. 
	We denote by $U$  the unit vector field on $(\mathbb{T}^{1}, g_{\mathbb{T}^{1}})$ in the fiber.  
	Let $\Ric_{(F,  e^{2 \alpha }g_{F})}$ and $\Ric_{(F,  g_{F})}$  be the Ricci curvatures
	of $(F,  e^{2\alpha }g_{F})$ and  $(F,  g_{F})$, respectively.   
	Then we have 
	\begin{align*}
	\Ric_{(F,  e^2{\alpha }g_{F})}(U,U)=& 
	\Ric_{(F,  g_{F})}(U,U)- (\dim F -2)(
	\Hess \alpha (U,U) - |d \alpha (U)|)  \\
	&+ \Delta_{(F,  g_{F})} \alpha - (\dim F -2)|  d \alpha|^{2} \\
	=&  \Delta_{(F,  g_{F})} \alpha - (\dim F -2)|  d \alpha|^{2}, \\
	\end{align*}
	where $\Delta_{(F,  g_{F})}$ is the Laplacian of  $(F,  g_{F})$.
	Since the left-hand side is positive, 
	\begin{equation*}
	\Delta_{(F,  g_{F})} \alpha \geq (\dim F -2)|  d \alpha|^{2} \geq 0.
	\end{equation*}
	Therefore, by the maximal principle, $\alpha$ is constant 
	on the fiber. 
	We have proved the case of $l=1$.  
	Therefore, Proposition~\ref{prop : str} has been proved.

	It is sufficient to prove Theorem~\ref{5}  for only  warped products.   
	From equation~(\ref{mean curv}) it follows that 
	for any lifts  $\widehat{U}$, $\widehat{V}$  of vector fields  $U$, $V$ on $F$, 
	\begin{equation*}
	T_{\widehat{U}} \widehat{V} = e^{2 \alpha } g_{F}(U,  V)  \nabla^{B} \alpha. 
	\end{equation*}
	It follows that $$\nu = \frac{ 1}{\sqrt{k}}  \nabla^{B} \alpha .$$ 
	Since $g(\nu, \nu)=-1$, we have 
	\begin{equation*}
	g_{B}(  \nabla^{B} \alpha,   \nabla^{B} \alpha) = k. 
	\end{equation*}
	Consider a function $$H=\frac{1}{\sqrt{k}}  \alpha : B \rightarrow \mathbb{R}. $$ 
	Then $H$ is a signed distance function as $g_{B}(\nabla^{B} H,   \nabla^{B} H) = 1$. 
	It follows that $B$ is diffeomorphic to $\mathbb{R} \times H^{-1}(0)$ since $(B, g_{B})$ is complete. 
	Let $\gamma_{0}$ be the geodesic with 
	$\gamma_{0}^{\prime} (s) =  - \nabla^{B} H (\gamma_{0}(s))$ in $(B, g_{B})$. 
	We can construct  incomplete lightlike and timelike geodesics  $\gamma (t) =(\gamma_{B}(t), \gamma_{F} (t))$ in $E$ by using the geodesic $\gamma_{0}$ in $B$. 
	First we consider lightlike geodesics. 
	By using Proposition~\ref{geodesic}, 
	lightlike geodesics satisfy the following equations
	\begin{align*}
	\nabla^{B}_{\partial/ \partial t} \gamma_{B}^{\prime} (t) &= 
	-  e^{2 \alpha}   g_{F}(\gamma_{F}^{\prime}(t),\, {\gamma_{F}}^{\prime}(t)) 
	\nabla^{B} \alpha  \\
	&= - g_{B}(\gamma_{B}^{\prime}(t), \gamma_{B}^{\prime}(t))  \nabla^{B} \alpha; \\
	\nabla^{F}_{\partial/\partial t} \gamma_{F}^{\prime}(t)
	&=-2 \frac{d\alpha (\gamma_{B}(t))}{dt} 
	{\gamma_{F}}^{\prime}(t), 
	\end{align*}
	since  $g_{B}(\gamma_{B}^{\prime}(t), \gamma_{B}^{\prime}(t)) =e^{2 \alpha}g_{F}(\gamma_{F}^{\prime}(t), \gamma_{F}^{\prime}(t))$. 
	Take a solution $s(t)$ of 
	\begin{equation}
	\frac{d^{2} s(t)}{dt^{2}} =\sqrt{k} \left( \frac{d s(t)}{dt} \right)^{2}. \label{eq : para}
	\end{equation} 
	Set $\gamma_{B} (t) = \gamma_{0} (s(t))$. 
	Since
	$\gamma_{B}^{\prime} ( t ) =  \gamma_{0}^{\prime} (s(t)) s^{\prime} (t)$, we have 
	$g_{B}(\gamma_{B}^{\prime}(t), \gamma_{B}^{\prime} ( t ))=s^{\prime} (t)^{2}$. 
	Therefore, we obtain 
	\begin{align*} 
	\nabla_{\partial/ \partial t }\gamma_{B}^{\prime} ( t ) &= s^{\prime \prime} (t)    \gamma_{0}^{\prime} (s(t)) + 
	\nabla_{\partial/ \partial s }\gamma_{0}^{\prime} ( s (t) ) (s^{\prime} (t))^{2} \\
	&=  s^{\prime \prime} (t)    \gamma_{0}^{\prime} (s(t))  \\
	&=  - \sqrt{k} (s^{\prime}(t) )^{2} \nabla^{B} H \\ 
	&= - g_{B}(\gamma_{B}^{\prime}(s),\gamma_{B}^{\prime}(s) ) \nabla^{B} \alpha \\
	\end{align*} 
	Thus, we have $\gamma_{B}(t)$, which means 
	we can also obtain $\gamma_{F}(t)$. 
	We have constructed a lightlike geodesic $(\gamma_{B}(t), \gamma_{F}(t) )$.  
	
	We solve 
	a solution $s (t)$ of  the differential equation~(\ref{eq : para}). 
	The solution is  $$s (t) = - \frac{1}{\sqrt{k}}\log |\sqrt{k} t + C_{1}| + C_{2}, $$ where $C_{1}$ and $C_{2}$ are integral constants.  
	Then $s( (-\frac{C_{1}}{\sqrt{k}}, \infty )) = \mathbb{R}$, and $$\lim_{t \to -\frac{C_{1}}{\sqrt{k}} +0} s(t)=\infty . $$    
	Therefore, the geodesic $(\gamma_{B}(t), \gamma_{F}(t) )$  is not defined on $\mathbb{R}$. 
	We see that the entire semi-Riemannian manifold is not lightlike geodesically complete. 
	
	Next we consider timelike geodesics. 
	By using Proposition~\ref{geodesic}, 
	timelike geodesics $\gamma(t)=(\gamma_{B}(t), \gamma_{F}(t))$ with  $g(\gamma^{\prime} (t), \gamma^{\prime} (t))=-1$
	satisfy the following equations
	\begin{align*}
	\nabla^{B}_{\partial/ \partial t} \gamma_{B}^{\prime} (t) &= 
	-  e^{2 \alpha}   g_{F}(\gamma_{F}^{\prime}(t),\, {\gamma_{F}}^{\prime}(t)) 
	\nabla^{B} \alpha  \\
	&=(1 - g_{B}(\gamma_{B}^{\prime}(t), \gamma_{B}^{\prime}(t)) ) \nabla^{B} \alpha; \\
	\nabla^{F}_{\partial/\partial t} \gamma_{F}^{\prime}(t)
	&=-2 \frac{d\alpha (\gamma_{B}(t))}{dt} 
	{\gamma_{F}}^{\prime}(t), 
	\end{align*}
	since  $-g_{B}(\gamma_{B}^{\prime}(t), \gamma_{B}^{\prime}(t)) +e^{2 \alpha}g_{F}(\gamma_{F}^{\prime}(t), \gamma_{F}^{\prime}(t))=-1$. 
	Assume that $g_{B}(\gamma_{B}^{\prime}(t), \gamma_{B}^{\prime}(t)) \not = 1$, that is 
	$\gamma_{F}^{\prime}(t) \not =  0$. 
	Take a solution $s(t)$ of 
	\begin{equation}
	\frac{d^{2} s(t)}{dt^{2}} =\sqrt{k}\left\{ \left( \frac{d s(t)}{dt}\right)^{2}-1 \right\}, \label{eq : para2}
	\end{equation} 
	where $\frac{d s(t)}{dt} > 1$. 
	Set $\gamma_{B} (t) = \gamma_{0} (s(t))$. 
	In the same manner as for the case of lightlike geodesics, 
	 we have 
	 \begin{equation*}
	\nabla_{\partial/ \partial t }\gamma_{B}^{\prime} ( t ) = (1- g_{B}(\gamma_{B}^{\prime}(s),\gamma_{B}^{\prime}(s) )) \nabla^{B} \alpha .
	\end{equation*} 
	Thus, we obtain $\gamma_{B}(t)$ and  $\gamma_{F}(t)$, which means we  
	have constructed a timelike geodesic $(\gamma_{B}(t), \gamma_{F}(t) )$.  
	We obtain the solution $s (t)$ of  the differential equation~(\ref{eq : para2}). 
	The solution is  $$s (t) = - \frac{1}{\sqrt{k}}\log | C_{1} e^{2 \sqrt{k}t} - 1 | + t + C_{2}, $$ where $C_{1} > 0 $ and $C_{2}$ are integral constants.  
	Note that $C_{1}>0$ follows from $\frac{d s(t)}{dt} > 1$. 
	Then $s( (- \infty , \frac{1}{2\sqrt{k}} \log \frac{1}{C_{1}}  )) = \mathbb{R}$, and $$\lim_{t \to \frac{1}{2\sqrt{k} } \log \frac{1}{C_{1}} -0 } s(t)=\infty . $$    
	Therefore, the geodesic $(\gamma_{B}(t), \gamma_{F}(t) )$  is not defined on $\mathbb{R}$. 
	We see that the entire semi-Riemannian manifold is not timelike geodesically complete. This is a contradiction.

	Hence, we have proved that the fiber never includes tori. 
	It follows that the fundamental group of the fiber is finite. 
	The proof of Theorem~\ref{5} is complete.

	\section{Semi-Riemannian example with non-integrable horizontal distribution}
	
	In this section, we construct a semi-Riemannian submersion such that 
	the entire space satisfies the curvature condition $R \geq k >0$ and that the horizontal distribution is not integrable. 
	Set $G = SU(2,1)=\left\{ g \in SL(2, \mathbb{C}) : 
	^{t} \overline{g} I_{2,1}
	g = I_{2,1}
	 \right\}$, where 
	 $I_{2,1}=
	  \begin{pmatrix}
	1 &0 & 0 \\
	0 & 1 & 0 \\
	0 & 0 & -1 
	\end{pmatrix}$. 
	Let $\mathfrak{g}$  be the Lie algebra of $G$. 
	The indefinite inner product $B(-,-)$ on $\mathfrak{g}$  is given  by
	\begin{equation*}
	B(X, Y)=- \Re \tr (XY).
	\end{equation*}
	This inner product is invariant under the adjoint action $Ad (g)$ for $g \in G$. 
	We write $K$ for a maximal compact subgroup   
	\begin{equation*}
	\left\{ 
	\begin{pmatrix}
	A & 0 \\
	0 & (\det A)^{-1}
	\end{pmatrix} :
	A \in U(2)
	\right\}.
	\end{equation*}
	Let $\mathfrak{k}$ be the Lie algebra of $K$. 
	Then  $B$ is positive definite on $\mathfrak{k}$ and negative definite on $\mathfrak{k}^{\perp}$.  
	Note that the homogeneous space $G/K$ with the induced metric from $-B$  is a complex hyperbolic plane $\mathbb{CH}^{2}$. 
        Write
	\begin{equation*}
	H=\left\{ 
	\begin{pmatrix}
	e^{2 \pi t i} & 0 & 0\\
	0 & e^{2 \pi t i} & 0\\
	 0 & 0 & e^{-4 \pi t i}\\
	\end{pmatrix} :
	t \in \mathbb{R}
	\right\} \subset K.
	\end{equation*}
	Note that $H$ is  the circle $S^{1}$.   
	Let $\mathfrak{h}_{0}$ be the Lie algebra of $H$. 
	We denote $\mathfrak{h}_{0}^{\perp} \cap \mathfrak{k}$ and $\mathfrak{k}^{\perp}$ 
	by $\mathfrak{h}_{1}$ and $\mathfrak{h}_{2}$, respectively. 
	Then we have $\mathfrak{g}=\mathfrak{h}_{0}\oplus\mathfrak{h}_{1} \oplus \mathfrak{h}_{2}$. 
	We write the projection of a vector $X \in  \mathfrak{g}$ onto the subspace
	$\mathfrak{h}_{j}$  as $X_{j}$ for each $j$. 
	Let us take the following basis of $\mathfrak{g}$: 
	\begin{align*}
	&e_{1}=
	\begin{pmatrix}
	i & 0 & 0 \\
	0 & i  & 0 \\
	0 & 0 & -2 i
	\end{pmatrix}, 
	e_{2}=
	\begin{pmatrix}
	i & 0 & 0 \\
	0 & -i  & 0 \\
	0 & 0 & 0
	\end{pmatrix}, 
	e_{3}=
	\begin{pmatrix}
	0 & 1 & 0 \\
	-1 & 0  & 0 \\
	0 & 0 & 0
	\end{pmatrix},
	e_{4}=
	\begin{pmatrix}
	0 & i & 0 \\
	i & 0  & 0 \\
	0 & 0 & 0
	\end{pmatrix}, \\  
	&f_{1}=
	\begin{pmatrix}
	0 & 0 & 1 \\
	0 & 0  & 0 \\
	1 & 0 & 0
	\end{pmatrix}, 
	f_{2}=
	\begin{pmatrix}
	0& 0 & 0 \\
	0 & 0  & 1 \\
	0 & 1 & 0
	\end{pmatrix}, 
	f_{3}=
	\begin{pmatrix}
	0 & 0 & i \\
	0 & 0  & 0 \\
	-i & 0 & 0
	\end{pmatrix},
	f_{4}=
	\begin{pmatrix}
	0 & 0 & 0 \\
	0 & 0  & i \\
	0 & -i & 0
	\end{pmatrix}. 
	\end{align*}
	Then $\mathfrak{h}_{0}$ is spanned by $e_{1}$, $\mathfrak{h}_{1}$ is spanned by $e_{2}$, $e_{3}$, $e_{4}$, and 
	$\mathfrak{h}_{2}$ is spanned by $f_{1}$, $f_{2}$, $f_{3}$, $f_{4}$. 
	Let $M$ be  a homogeneous space $G/H$. 
	We define the indefinite inner metric $(-,-)$ on $\mathfrak{g}/\mathfrak{h}_{0}$ by $(X,Y)= (1+t)B(X_{1},Y_{1})+B(X_{2},Y_{2})$ for 
	$X,Y \in \mathfrak{g}/\mathfrak{h}_{0}$ and can also define the semi-Riemannian metric $(-,-)$ on $G/H$, which is left-invariant under $G$.  
	We see that the projective map $\pi : M \rightarrow \mathbb{CH}^{2} $ is a semi-Riemannian submersion. 
	
	We will show that $M$ satisfies the curvature condition $R \geq k>0$.    
	We denote the curvature tensor of $(M, (-,-))$ by $R$.    
	Note that $[\mathfrak{h}_{1}, \mathfrak{h}_{1}]\subset \mathfrak{h}_{1}$,  
	$[\mathfrak{h}_{2}, \mathfrak{h}_{2}]\subset \mathfrak{h}_{0}\oplus\mathfrak{h}_{1}$,  
	and $[\mathfrak{h}_{1}, \mathfrak{h}_{2}]\subset \mathfrak{h}_{2}$.   
	Then we have 
	\begin{align*}
	(R(X,Y)Y, X)=&
	\frac{1-3t}{4}B([X,Y]_{1}, [X,Y]_{1})
	+(t-t^{2})B([X_{1},Y_{1}], [X,Y]) \\
	&+t^{2}B([X_{1},Y_{1}], [X_{1},Y_{1}])
	+\frac{(1+t)^{2}}{4}B([X,Y]_{2}, [X,Y]_{2}) \\
	&+B([X,Y]_{0}, [X,Y]_{0}) \\
	=&
	\frac{1+t}{4}B([X_{1},Y_{1}], [X_{1},Y_{1}])
	+\frac{1-3t}{4}B([X_{2},Y_{2}]_{1}, [X_{2},Y_{2}]_{1}) \\
	&+\frac{1-t-2t^{2}}{2}B([X_{1},Y_{1}], [X_{2},Y_{2}]_{1})
	+\frac{(1+t)^{2}}{4}B([X,Y]_{2}, [X,Y]_{2}) \\
	&
	+B([X,Y]_{0}, [X,Y]_{0}).
	\end{align*}
	Set $X=a_{2}e_{2}+a_{3}e_{3}+a_{4}e_{4}+b_{1}f_{1}+b_{2}f_{2}+b_{3}f_{3}+b_{4}f_{4}$, 
	$Y=c_{2}e_{2}+c_{3}e_{3}+c_{4}e_{4}+d_{1}f_{1}+d_{2}f_{2}+d_{3}f_{3}+d_{4}f_{4}$. 
	Then, straightforward computation implies 
	\begin{align*}
	B([X,Y]_{2}, [X,Y]_{2}) =& 
	-2\biggl\{ 
	\left( \det 
	\begin{pmatrix}
	a_{3} & c_{3} \\
	b_{2} & d_{2}
	\end{pmatrix}
	- 
	 \det 
	\begin{pmatrix}
	a_{2} & c_{2} \\
	b_{3} & d_{3}
	\end{pmatrix}
	-
	\det 
	\begin{pmatrix}
	a_{4} & c_{4} \\
	b_{4} & d_{4}
	\end{pmatrix}
	\right)^{2} \\
	&+ 
	\left( \det 
	\begin{pmatrix}
	a_{2} & c_{2} \\
	b_{4} & d_{4}
	\end{pmatrix}
	- 
	 \det 
	\begin{pmatrix}
	a_{3} & c_{3} \\
	b_{1} & d_{1}
	\end{pmatrix} 
	-
	\det 
	\begin{pmatrix}
	a_{4} & c_{4} \\
	b_{3} & d_{3}
	\end{pmatrix}
	\right)^{2} \\
	&+
	\left( \det 
	\begin{pmatrix}
	a_{2} & c_{2} \\
	b_{1} & d_{1}
	\end{pmatrix}
	+ 
	 \det 
	\begin{pmatrix}
	a_{3} & c_{3} \\
	b_{4} & d_{4}
	\end{pmatrix}
	+
	\det 
	\begin{pmatrix}
	a_{4} & c_{4} \\
	b_{2} & d_{2}
	\end{pmatrix}
	\right)^{2} \\
	&+
	\left( \det 
	\begin{pmatrix}
	a_{2} & c_{2} \\
	b_{2} & d_{2}
	\end{pmatrix}
	+ 
	 \det 
	\begin{pmatrix}
	a_{3} & c_{3} \\
	b_{3} & d_{3}
	\end{pmatrix} 
	-
	\det 
	\begin{pmatrix}
	a_{4} & c_{4} \\
	b_{1} & d_{1}
	\end{pmatrix}
	\right)^{2}  \biggr\} \\
	=&-2 \sum_{i,j}\left( 
	\det
	\begin{pmatrix}
	a_{i} & c_{i} \\
	b_{j} & d_{j}
	\end{pmatrix}\right)^{2}
	+B([X_{1},Y_{1}], [X_{2},Y_{2}]_{1}). 
	\end{align*}
	By the Cauchy--Schwarz inequality, we have 
	\begin{equation*}
	B([X_{1},Y_{1}], [X_{2},Y_{2}]_{1}) \leq 
	\sqrt{B([X_{1},Y_{1}], [X_{1},Y_{1}] ) B( [X_{2},Y_{2}]_{1} ,  [X_{2},Y_{2}]_{1}) }.
	\end{equation*}
	Set 
	\begin{align*} 
	x= \sqrt{\sum_{i<j}\left(\det 
	\begin{pmatrix}
	a_{i} & c_{i} \\
	a_{j} & c_{j}
	\end{pmatrix}
	\right)^{2}}, 
	y=\sqrt{ \sum_{i<j}\left(\det 
	\begin{pmatrix}
	b_{i} & d_{i} \\
	b_{j} & d_{j}
	\end{pmatrix}
	\right)^{2} },
	z= \sqrt{\sum_{i , j}\left( \det 
	\begin{pmatrix}
	a_{i} & c_{i} \\
	b_{j} & d_{j}
	\end{pmatrix}
	\right)^{2} }.
	\end{align*}
	We should remark that
	\begin{equation*}
	B([X_{1},Y_{1}], [X_{1},Y_{1}] )=8 x^{2},\quad  B( [X_{2},Y_{2}]_{1} ,  [X_{2},Y_{2}]_{1})=2y^{2}.
	\end{equation*}
	Then we have
	\begin{align*}
	(R(X,Y)Y, X) &\geq
	2(1+t) x^{2}
	+\frac{1-3t}{2} y^{2} -\frac{(1+t)^{2}}{2} z^{2}
	- \left| \frac{1-t-2t^{2}}{2} +\frac{(1+t)^{2}}{4}\right| \sqrt{16x^{2}y^{2}} \\
	&=2(1+t)x^{2}-3|1-t^{2}|xy+\frac{1-3t}{2}y^{2}-\frac{(1+t)^{2}}{2}z^{2}.
	\end{align*} 
	Note that
	\begin{align*}
	(X,X)(Y,Y)-(X,Y)^{2}=4(1+t)^{2}x^{2}+4 y^{2} - 4(1+t)z^{2}.
	\end{align*}
	Therefore, we obtain 
	\begin{align*}
	&(R(X,Y)Y, X)-k( (X,X)(Y,Y)-(X,Y)^{2}) \\
	&\geq 
	\{2(1+t)-4k (1+t)^{2} \}x^{2}-3|1-t^{2}|xy+\left\{ \frac{1-3t}{2} - 4k \right\}y^{2}+ \left(4k(1+t)- \frac{(1+t)^{2}}{2}  \right) z^{2}.
	\end{align*}
	Hence, it is sufficient to prove that there exists $k>0$ and $t > -1$ such that 
	\begin{equation*}
	\{2(1+t)-4k (1+t)^{2} \}x^{2}-3|1-t^{2}|xy+\left\{ \frac{1-3t}{2} - 4k \right\}y^{2}+ \left\{4k(1+t)- \frac{(1+t)^{2}}{2}  \right\} z^{2}
	> 0. 
	\end{equation*}
	Therefore, 
	our search is reduced to finding $k>0$ and $t > -1$ satisfying  the following four inequalities:  
	\begin{gather}
	k>\frac{1+t}{8}, \label{ineq1}\\ 
	k< \frac{1}{2(1+t)}, \label{ineq2}\\ 
	k<\frac{1-3t}{8}, \label{ineq3} \\
	\{2(1+t)-4k (1+t)^{2} \}	\left( \frac{1-3t}{2} - 4k \right) -\frac{9}{4}(1-t^{2})^{2} >0. \label{ineq4}
	\end{gather}
	From the inequalities (\ref{ineq1}), (\ref{ineq2}), and (\ref{ineq3}), we see that 
	\begin{equation*}
	\frac{1+t}{8} < k < \frac{1-3t}{8} < \frac{1}{2(1+t)},
	\end{equation*}
	for $-1 < t <0$. 
	The inequality (\ref{ineq4}) is  
	\begin{equation*}
	(t+1)\left\{16(1+t)k^{2}+2(3 t^{2} +2 t -5) k - \frac{1}{4}(9 t^{3} - 9t^{2}+3t+5 )\right\}>0. 
	\end{equation*} 
	Set 
	$$\eta (t)=\frac{-3t^{2} -2 t +5 -\sqrt{45 t^{4}  +12 t^{3} -50 t^{2} +12 t +45}  }{16(t+1)}.$$ 
	We note that $\eta (t)$ is a solution of the equation with respect to $k$  
	 that the left-hand side of the  
	inequality (\ref{ineq4}) is zero. 
	Then  the inequalities  (\ref{ineq1}), (\ref{ineq2}),  (\ref{ineq3}), and  (\ref{ineq4})  hold for $-1 < t < -\frac{3}{5}$ and 
	$\frac{1+t}{8} < k < \eta (t)$.

	 We will prove that $M$ is geodesically complete. 
	 Since a natural projection $G \rightarrow M$ is a semi-Riemannian submersion, 
	 we consider only geodesics $\gamma(u)$ in $G$ from the identity element  
 	  of which the initial velocity belongs to $\mathfrak{h}_{1}\oplus \mathfrak{h}_{2}$. 
	  Let $\Gamma(u)$ be a curve $\gamma(u)^{-1}\gamma^{\prime}(u)$ in $\mathfrak{h}_{1} \oplus \mathfrak{h}_{2}$. 
	  We denote by $\Gamma_{j} (u)$ the $\mathfrak{h}_{j}$-component of $\Gamma(u)$. 
	  Set a linear operator $\phi : \mathfrak{g} \rightarrow  \mathfrak{g}$ such that 
	  $\phi (X_{0}+X_{1}+X_{2})=X_{0}+ (1+t)X_{1}+X_{2}$ for any $X_{j} \in \mathfrak{h}_{j}$. 
	  Note that we have $(X,Y)=B(\phi(X), Y)$. 
	 Then $\gamma(u)$ satisfies the following Euler--Arnold equation: 
	 \begin{equation*}
	 \phi( \Gamma^{\prime}(u)) = [\phi(\Gamma (u)), \Gamma(u)] 
	 \end{equation*}
	 (for instance see  \cite{MR1310948}). 
	 Therefore, we have 
	 \begin{align}
	 &(1+t) \Gamma_{1}^{\prime}(u)=0, \label{eq:41}
	  \\ &\Gamma_{2}^{\prime}(u)=t[\Gamma_{1}(u), \Gamma_{2}(u)]. \label{eq:42}
	  \end{align}
	  The equation (\ref{eq:41}) implies $\Gamma_{1}(u)$ is constant, denoted by $v_{1}$.  
	  Then, the equation (\ref{eq:42}) means  
	  $\Gamma_{2}^{\prime}(u)=t[v_{1},  \Gamma_{2}(u)]$, which is a system of first-order linear differential 
	  equations with a constant coefficient. 
	  It follows that $\gamma(u)$  is defined in the real line $\mathbb{R}$. 
	  Hence, the semi-Riemannian manifold $M$ we have constructed is the desired solution.  
	 
	\providecommand{\bysame}{\leavevmode\hbox to3em{\hrulefill}\thinspace}
	\providecommand{\MR}{\relax\ifhmode\unskip\space\fi MR }
	\providecommand{\MRhref}[2]{%
		\href{http://www.ams.org/mathscinet-getitem?mr=#1}{#2}
	}
	\providecommand{\href}[2]{#2}
	
\end{document}